\documentclass{amsart}
\usepackage[utf8]{inputenc}
\usepackage{color}
\usepackage[english]{babel}
\usepackage{amsmath} 
\usepackage{amsfonts}
\usepackage{amssymb}
\usepackage{epigraph}
 \usepackage{graphicx}

\usepackage{setspace}
\usepackage[utf8]{inputenc}
\usepackage[T1]{fontenc}
\usepackage{graphicx}
\usepackage{amsthm}
\usepackage{mathrsfs}
\usepackage{bm}
\usepackage[all]{xy}
\usepackage{lscape}
\usepackage{latexsym}
\usepackage{epigraph}
\usepackage[Sonny]{fncychap}
\usepackage{amsfonts}

\theoremstyle{plain}
\newtheorem{theorem}{Theorem}

\newtheorem{lemma}{Lemma}
\newtheorem{corollary}{Corollary}

\theoremstyle{definition}

\theoremstyle{remark}

\definecolor{darkred}{rgb}{1,0,0} 
\definecolor{darkgreen}{rgb}{0,1,0}
\definecolor{darkblue}{rgb}{0,0,1}

\newcommand{\K}{\mathfrak{k}}

\definecolor{darkred}{rgb}{1,0,0} 
\definecolor{darkgreen}{rgb}{0,0.8,0}
\definecolor{darkblue}{rgb}{0,0,1}

\begin{document}

\title{On Killers of Cable  Knot Groups}

\author{Ederson R. F.  Dutra}
\address{Mathematisches Seminar, Christian-Albrechts-Universität zu Kiel, Ludewig-Meyn Str.~4, 24098 Kiel,
Germany}
\email{dutra@math.uni-kiel.de}
\thanks{ 
This work was supported by CAPES - Coordination for the Improvement of Higher Education Personnel-Brazil,  program Science Without Borders grant 13522-13-2}

\begin{abstract}
A killer of a group $G$ is an element that normally generates $G$.  We show that the group of a cable knot contains infinitely many killers such that no two lie in the same automorphic orbit.
\end{abstract}

\maketitle

\section{Introduction}

Let $G$ be an arbitrary group and $S\subseteq G$. We define the normal closure $\langle\langle S\rangle\rangle_G$ of $S$  as the smallest normal subgroup of $G$ containing $S$,   equivalently   
$$\langle\langle S\rangle\rangle_{G}= \{\prod_{i=1}^{k}u_is_i^{\varepsilon_i}u_i^{-1} \ | \ u_i\in G, \varepsilon_i=\pm 1, s_i\in S, k\in \mathbb{N}\}.$$ 
Following \cite{Simon}, we call an element $g\in G$ a \textit{killer} if $\langle\langle g\rangle\rangle_{G}=G$ . We say that two killers $g_1,g_2\in G$  are \textit{equivalent} if there exists an automorphism  $\phi:G\rightarrow G$  such that $\phi(g_1)=g_2$. 
 
Let $\K$ be a knot in $S^3$ and $V(\K)$ a regular neighborhood of $\K$. Denote by  
$$X(\K)= S^3-Int(V(\K)) $$ the knot manifold  of $\K$  and by   $G(\K)=\pi_1(X(\K))$ its group.  A \emph{meridian} of $\K$ is an element of $G(\K)$ which can be represented by a simple closed curve  on  $ \partial V(\K)$  that is contractible in $V(\K)$ but not contractible  in $\partial V(\K)$.  Thus a meridian is well defined up to conjugacy and inversion.

From a Wirtinger presentation of $G(\K)$ we see  that the meridian is a killer.  In \cite[Theorem 3.11]{Tsau} the author  exhibit    a knot for which there exists  a killer  that is not equivalent to the meridian. Silver–Whitten–Williams  \cite[Corollary 1.3]{Silver}  showed that if $\K$ is  a hyperbolic $2$-bridge knot or a  torus knot or  a    hyperbolic knot with unknotting number one, then its group contains infinitely many  pairwise inequivalent  killers. 

In \cite[Conjecture 3.3]{Silver} it is conjectured  that the group of  any nontrivial knot has infinitely many inequivalent killers, see also   \cite[Question 9.26]{Friedl}.  In this paper we show the following.

\begin{theorem}{\label{T1}}
 Let $\K$ be a cable knot  about a nontrivial knot $\K_1$. Then its group contains infinitely many pairwise inequivalent killers.
\end{theorem}

Moreover, we show that having infinitely many inequivalent killers is preserved under connected sums. As a Corollary we show that the group of any nontrivial knot whose exterior is a graph manifold contains infinitely many inequivalent killers. 

\section{Proof of Theorem 1}
Let $m,n$ be coprime  integers    with $n\geq 2$. The \emph{cable space} $CS({m,n})$ is defined  as follows: let $D^2=\{z\in \mathbb{C} \ | \ \Vert z\Vert \leq 1\}$ and $\rho :D^2\rightarrow D^2$ a rotation through an angle of $2\pi (m/n)$ about the origin. Choose a disk  $\delta\subset Int(D^2)$ such that $\rho^i(\delta)\cap \rho^j(\delta)=\emptyset$ for $1\leq i\neq j\leq n$ and  denote by  $D^2_n$   the space
$$D^2-Int\Big(\bigcup_{i=1}^n\rho^i(\delta)\Big). $$
 $\rho$ induces a homeomorphism $\rho_0:=\rho|_{D^2_n}:D_n^2\rightarrow D_n^2$. We define $CS(m,n)$ as the mapping torus of $\rho_0$, i.e. 
$$CS(m,n):=D^2_n\times I/(z,0)\sim (\rho_0(z),1).$$
Note that $CS(m,n)$ has the structure of a Seifert fibered space. Each fiber is  the image of $\{\rho^i(z)|1\leq i\leq n\} \times I$ under the quotient map, where $z\in D^2_n$.  There is exactly one exceptional fiber, namely the image $C_0$ of the arc $0\times I$.

In order to compute the fundamental group $A$ of $CS(m,n)$,  denote the free generators of $\pi_1(D_n^2)$ corresponding to  the boundary  paths of the removed disks $\rho_0(\delta),\ldots, \rho_0^n(\delta)$  by $x_1,\ldots,x_n$ respectively. From the definition of $CS(m,n)$ we see that we can write $A$ as the semi-direct product $F(x_1,\ldots,x_n)\rtimes \mathbb{Z}$, where the action of $\mathbb{Z}=\langle t\rangle$ on $\pi_1(D_n^2)=F(x_1,\ldots,x_n)$ is given by 
$$tx_it^{-1}=x_{\sigma(i)} \text{ for } 1\leq i\leq n.$$
The element $t$ is represented by the exceptional fiber of $CS(m,n)$ and the permutation  $\sigma:\{1,\ldots,n\}\rightarrow \{1,\ldots, n\}$ is given by $i  \mapsto i+m\text{ mod } n$. 
Thus  
$$ A=\langle x_1,\ldots,x_n, t \ | \ tx_it^{-1}=x_{\sigma(i)} \ \text{for} \ 1\leq i\leq n\rangle.$$ 
We finally remark that  any element  $a\in A$ is uniquely written as $w\cdot t^z$ for some $w\in F(x_1,\ldots, x_n)$ and $z\in \mathbb{Z}$.

We next define cable knots. Let $V_0$ be the solid torus
 $D^2\times I/(z,0)\sim (\rho(z),1) $ 
and for some $z_0\in Int(D^2)-0$ let  $\K_0$ be   the image of $\{\rho^i(z_0)\hspace{0.5mm}|\hspace{0.5mm}1\leq i\leq n\}\times I$ under the quotient map. Note that $\K_0$ is a simple closed curve contained in the interior of $V_0$. 
Let $\K_1$ be a nontrivial knot in $S^3$ and $V(\K_1)$ a regular neighborhood of $\K_1$ in $S^3$.  Let further $h:V_0\rightarrow V(\K_1)$ be a homeomorphism which maps  the meridian $\partial D^2\times 1$ of $V_0$ to a meridian of $\K_1$.  The knot $\K:=h(\K_0)$  is called a $(m,n)$\emph{-cable knot} about $\K_1$.

Thus the knot manifold $X(\K)$ of a $(m,n)$-cable knot $\K$ decomposes as
$$ X(\K)=CS({m,n})\cup X(\K_1)$$
with $\partial X(\K_1)=CS({m,n})\cap X(\K_1) $ an incompressible torus in $X(\K)$.  It follows from the Theorem of Seifert and van-Kampen that 
$$G(\K)=A\ast_{C}B$$ 
where $B=G(\K_1)$ and  $C=\pi_1 (\partial X(\K_1))$. Denote by $m_1$ the meridian of $\K_1$ and note that in $A$ we have $m_1=x_1\cdot \ldots\cdot  x_n$. In turn, the meridian $m\in G(\K)$ of $\K$ is written as  $m=x_1\in A$.  

The proof of Theorem~\ref{T1} is  divided in two steps. In Lemma~\ref{L1} we   exhibit elements that normally generate the  group  of the cable knot  and next, in Lemma~\ref{L2},  we prove that these killers are indeed inequivalent.
 
Choose    $s\in \{1,\ldots,n-1\}$ such that $\sigma^s(1)=2$. Since  $\sigma^s(i)=i+sm \text{ mod }n$ it follows that $\sigma^s= (1 \ \ 2  \ \ 3 \ \ldots  \ n-1 \ \ n)$.

\begin{lemma}\label{L1}
Let $\K $ be a $(m,n)$-cable knot about a nontrivial knot $\K_1 $. Then  for each $l\geq 1$ the element 
$$g_l: =x_{1}^lx_2^{-(l-1)}=x_1^l\cdot (t^sx_1t^{-s})^{-(l-1)}$$ 
normally generates the group of $\K $.
\end{lemma}
\begin{proof}

The first step of the proof is to show that the group of the companion knot is contained in $\langle\langle g_l\rangle\rangle_{G(\K)}$. 

\textbf{Claim 1:} The meridian $m_1=x_1 \cdot\ldots\cdot  x_n$ of $\K_1$ belongs to   $ \langle\langle g_l\rangle\rangle_{G(\K)} $. Consequently $B=\langle\langle m_1\rangle\rangle_{B}=\langle\langle m_1\rangle\rangle_{G(\K)}\cap B  \subseteq\langle\langle g_l \rangle\rangle_{G(\K)} $.

Note that for $0\leq i\leq n-1$ we have $t^{is}g_lt^{-is}=x_{i+1}^lx_{i+2}^{-(l-1)}$, where  indices are taken  mod $n$. Thus
\begin{eqnarray}\nonumber
x_1\cdot \ldots\cdot x_n &=& x_1^{-(l-1)}(x_1^lx_2\cdot \ldots\cdot x_nx_1^{-(l-1)})x_1^{l-1} \\ \nonumber
              &=& x_1^{-(l-1)} \Big(\prod_{i=0}^{n-1}x_{i+1}^lx_{i+2}^{-(l-1)}\Big)x_1^{l-1} \\ \nonumber 
              &=& x_1^{-(l-1)}\Big(\prod_{i=0}^{n-1} t^{is}\cdot g_l\cdot  t^{-is}\Big)x_1^{l-1}\\ \nonumber
              &=& \prod_{i=0}^{n-1}\Big(x_n^{-(l-1)} t^{is} \cdot g_l\cdot t^{-is}x_1^{l-1}\Big)\nonumber\\
              & = & \prod_{i=0}^{n-1}\Big( (x_1^{-(l-1)}t^{is})\cdot g_l \cdot (x_1^{(-l-1)}t^{is})^{-1}\Big)\nonumber
\end{eqnarray}
which implies that $m_1\in \langle\langle g_l\rangle\rangle_{G(\K)} $.  Thus Claim 1 is proved.  

From Claim 1 it follows that the peripheral subgroup $C=\pi_1(\partial X(\K_1))$ of $\K_1$ is contained in $\langle\langle g_l\rangle\rangle_{G(\K)}$ since $C\subseteq B$ and consequently we have 
$$G(\K)/\langle\langle g_l\rangle\rangle_{G(\K)} =A\ast_{C}B/\langle\langle g_l\rangle\rangle_{G(\K)} \cong A/\langle\langle g_l , C\rangle\rangle_A. $$
Thus,   we need to show that   $A/\langle\langle g_l, C\rangle\rangle_A=1$. It is easy to see that $A/\langle\langle C\rangle\rangle_A$ is cyclically  generated by $\pi(x_1)$, where $\pi:A\rightarrow  A/\langle\langle C\rangle\rangle_A$ is the canonical projection.  The result now follows from the fact that $\pi(g_l)=\pi(x_1^l\cdot t^sx_1^{(l-1)}t^{-s})=\pi(x_1)$.
\end{proof}

\begin{lemma}\label{L2}
If  $k\neq l$, then $g_k$ is not equivalent to $g_l$. \end{lemma}
\begin{proof} Assume that  $\phi:G(\K)\rightarrow G(\K)$ is an automorphism such that $\phi(g_l)=g_k$ and let $f:X(\K)\rightarrow X(\K)$ be a homotopy equivalence inducing $\phi$. From \cite[Theorem 14.6]{Johannson}  it follows that  $f$ can be deformed into $\hat{f}:X(\K)\rightarrow X(\K)$ so that  $\hat{f}$ sends $X(\K_1)$ homeomorphically onto $X(\K_1)$ and $\hat{f}|_{CS(m,n)}:CS(m,n)\rightarrow CS(m,n)$ is a homotopy equivalence. Thus $\phi(A)$ is conjugated to $A$, that is,   $\phi(A)=gAg^{-1}$ for some $g\in G(\K)$.  Since $\phi(g_l)=g_k$  it implies that $g_k\in gAg^{-1}$. As $g_k$ is not conjugated (in $A$) to an element of $C$, it implies that $g\in A$ and so $\phi(A)=A$. By \cite[Proposition 28.4]{Johannson}, we may assume that $\hat{f}|_{CS(m,n)}$ is fiber preserving. Since $CS(m,n)$ has exactly one exceptional fiber, which represents $t$,  we must have $\phi(t)=at^{\eta }a^{-1}$ for some $a=v\cdot t^{z_1}\in A$ and some $\eta\in \{\pm 1\}$.

The automorphism $\phi|_{A}:A\rightarrow A$ induces an automorphism $\phi_{\ast}$ on the factor group $A/\langle t^n\rangle=\langle x_1,t \ | \ t^n=1\rangle=\mathbb{Z}\ast \mathbb{Z}_n$ such that $\phi_{\ast}(t)=at^{\eta}a^{-1}$. It is a standard fact about automorphisms of free products that we must have  ${\phi}_{\ast}(x_1)=at^{e_0}x_1^{\varepsilon}t^{e_1}a^{-1}$ for $e_0,e_1\in \mathbb{Z}$ and $\varepsilon \in \{\pm 1\}$. Thus,
 $$\phi(x_1)=at^{e_0}x_1^{\varepsilon} t^{e_1}a^{-1}t^{dn}=at^{e_0}\cdot x_1t^{e_0+e_1+dn}\cdot t^{-e_0}a^{-1}$$ for some $d\in \mathbb{Z}$. Since $t$ has non-zero homology in $H_1(X(\K))$ it follows that $e_0+e_1+dn=0$. Consequently,    $\phi(x_1)= b \cdot x_1^{\varepsilon}\cdot b^{-1}$, where  $b=at^{e_0}=v\cdot t^{z_2}\in A$ and  $z_2=z_1+e_0$.

Hence  we obtain  
\begin{eqnarray}
\phi(g_l) & = & \phi(x_1^lx_2^{-(l-1)})\nonumber \\
          & = & \phi(x_1^l\cdot  t^{s}x_1^{-(l-1)}t^{-s})\nonumber \\
          & = & bx_1^{\varepsilon l}b^{-1}\cdot at^{ \eta s}a^{-1}\cdot  bx_1^{-\varepsilon(l-1)}b^{-1}\cdot at^{-\eta s}a^{-1}\nonumber \\
          & = & vt^{z_2}x_1^{\varepsilon l}t^{-z_2}v^{-1}\cdot vt^{z_1}t^{\eta s}t^{-z_1}v^{-1}\cdot vt^{z_2}x_1^{-\varepsilon (l-1)}t^{-z_2}v^{-1}\cdot vt^{z_1}t^{-\eta s}t^{-z_1}v^{-1} \nonumber\\
          & = & vx_i^{\varepsilon l}   x_j^{-\varepsilon (l-1) }v^{-1} \nonumber
\end{eqnarray}
where $i=\sigma^{z_2}(1)$ and $j=\sigma^{z_2+\eta s}(1)$. Note that $i\neq j$ since $\sigma^{s}(1)=2$ and $\sigma^{-s}(1)=n$.  Hence,  $\phi(g_l)=g_k$ implies that  
$$ v(x_i^{\varepsilon l}\cdot  x_j^{-\varepsilon (l-1) })v^{-1} = x_1^k\cdot x_2^{-(k-1)}$$
in   $F(x_1,\ldots, x_n)$. Thus, in the abelinization  of $F(x_1,\ldots,x_n)$ we have
$$\varepsilon[lx_i+(1-l)x_j]=kx_1+(1-k)x_2$$ 
which implies that $\{i,j\}=\{1,2\}$. If $(i,j)=(1,2)$, then $\varepsilon l=k$ and so $k=|k|=|\varepsilon l|=l$. If $(i,j)=(2,1)$, then 
 $ \varepsilon l=k-1$ and 
 $\varepsilon(1-l)=k$.
Consequently, $\varepsilon=1$ and   $l+k=1$ which is impossible  since $k,l\geq 1$.  
 \end{proof}

\section{Connected sums and killers}
In this section we show that having infinitely many inequivalnet killers is preserved under connected sums of knots. This fact, Theorem~\ref{T1}, and Corollary 1.3 of \cite{Silver} imply that the group of knots whose exterior is a graph manifold have infinitely many inequivalent killers. 
 
Let $\K$ be a knot and $\K_1,\ldots, \K_n$ its prime factors, that is, $\K= \K_1\sharp \ldots \sharp \K_n$ and each $\K_i$ is a nontrivial prime knot.  Assume that $x\in G(\K_i)$ is a killer of $G(\K_i)$.  It is well-known that  $G(\K_i)\leq G(\K)$  and $ \langle m\rangle \leq G(\K_i)$ for all $i$, where $m$ denotes the meridian of $\K$.   From this we immediately see that $m\in \langle\langle x\rangle\rangle_{G(\K_i)} \subseteq \langle\langle x\rangle\rangle_{G(\K)}$ which implies that $G(\K)=\langle \langle m\rangle\rangle_{G(\K)}\subseteq \langle\langle x\rangle\rangle_{G(\K)}$, ie., $x$ is a killer of $G(\K)$.

Now suppose that $x,y\in G(\K_i)$ are  killers of $G(\K_i)$ and  that there exists an automorphism   $\phi$  of $G(\K)$ such that $\phi(x)=y$. $\phi$ is induced by a homotopy equivalence   $f:X(\K)\rightarrow X(\K)$.   From \cite[Theorem 14.6]{Johannson} it follows that  $f$ can be deformed into $\hat{f}:X(\K)\rightarrow X(\K)$ so that: 
\begin{enumerate}
\item[1.] $\hat{f}{|_V}:V\rightarrow V$ is a homotopy equivalence, where $V=S^1\times(n-\text{punctured disk})$ is the peripheral component of the characteristic submanifold of $X(\K)$. 
\item[2.] $\hat{f}|_{\overline{X(\K)-V}}  :\overline{X(\K)-V}\rightarrow \overline{X(\K)-V}$ is a homeomorphism. 
\end{enumerate}   
Note that $\overline{X(\K)-V}=X(\K_1)\mathop{\dot{\cup}}\ldots \mathop{\dot{\cup}} X(\K_n)$. Since $\hat{f}|_{\overline{X(\K)-V}}$  is a homeomorphism it follows that  $\hat{f}$ sends $X(\K_i)$ homeomorphically onto $X(\K_{\tau(i)})$ for some permutation $\tau$ of $\{1,\ldots,n\}$. Consequently,  there exists $g^{\prime}\in G(\K)$ such that  
$$\phi(G(\K_i))=g^{\prime}G(\K_{\tau(i)}){g^{\prime}}^{-1}.$$ 
If $\tau(i)=i$ and $g'\in G(\K_i)$, then $\phi$ induces an automorphism $\psi:= \phi|_{G(\K_i)}$ of $G(\K_i)$ such that $\psi(x)=y$, i.e., $x$ and $y$ are equivalent in $G(\K_i)$.   If $\tau(i)\neq i$ or $g'\notin G(\K_{\tau(i)})$, then it is not hard to see that $y$ is conjugated (in $G(\K_i)$) to an element of $\langle m\rangle$ since $y=\phi(x)\in G(\K_i)\cap g^{\prime}G(\K_{\tau(i)}){g^{\prime}}^{-1}$. As $\langle\langle m^k\rangle\rangle\neq G(\K)$ for $|k|\geq 2$ and $y$ normally generates $G(\K)$,  we conclude that $y$ is conjugated (in $G(\K_i))$) to $m^{\pm1}$. The same argument applied to $\phi^{-1}$ shows that $x$ is conjugated (in $G(\K_i))$) to $m^{\pm1}$. 

Therefore, if the group of one of the prime factors of $\K$ has infinitely many inequivalent killers, then so does the group of $\K$.  As a Corollary of Theorem~\ref{T1} and the remark made above  we obtain the following result.

\begin{corollary}
If  $\K$ is  a knot such that $X(\K)$ is a graph manifold,  then $G(\K)$ contains  infinitely many pairwise inequivalent killers.
\end{corollary} 
\begin{proof}
From  \cite{JS}  it follows that  the only Seifert-fibered  manifolds that can be embedded into  a knot manifold with incompressible boundary   are torus knot complements, composing spaces and cable spaces. Thus, if $X(\K)$ is a graph manifold, then  one of the following holds:
\begin{enumerate}
\item  $\K$ is a torus knot.
\item  $\K$ is a cable knot.
\item  $\K=\K_1\sharp \ldots \sharp \K_n$ where each $\K_i$  is either a  torus knot or a cable knot.
\end{enumerate}
Now the result  follows from Theorem~\ref{T1} and  \cite[Corollary 1.3]{Silver}.
\end{proof}



\begin{thebibliography}{99}

\bibitem{Friedl} M. Aschenbrenner, S. Friedl, and H. Wilton, {$3$-manifold groups},   arXiv:1205.0202v3 [math.GT].

\bibitem{JS} W.~Jaco and P.~Shalen, {\it Seifert fibered spaces in 3-manifolds}, Mem. AMS \textbf{220}(1979).

\bibitem{Johannson} K.~Johanson, {\it Homotopy Equivalences of 3-manifolds with boundary}, Lectures Notes in Math., Vol. 761, Springer-Verlag, Berlin and New York, 1979.
 
 
\bibitem{Silver} D. S. Silver, W. Whitten, S. G. Williams, \textit{Knot groups with many killers}, Bull. Aust. Math. Soc. \textbf{81} (2010), 507-513. 

\bibitem{Simon} J. Simon, \textit{Wirtinger approximations and the knot groups of $F^n$ in $S^{n+1}$}, Pacific J. Math \textbf{90} (1990), 177-189. 


\bibitem{Tsau} C. M. Tsau, \textit{Nonalgebraic killers of knot groups}, Proc. Amer. Math. Soc. \textbf{95} (1985), 139-146.

\end{thebibliography}
\end{document}